\documentclass[11pt]{article}
\usepackage[margin=1.1in]{geometry}
\usepackage{amsmath,amssymb,amsthm}
\usepackage{microtype}
\usepackage{needspace}
\usepackage[colorlinks=true,linkcolor=blue,citecolor=blue,urlcolor=blue]{hyperref}
\usepackage{tikz}
\usepackage{float}
\usepackage{graphicx}
\usepackage{caption}
\usetikzlibrary{positioning,calc}
\newtheorem{theorem}{Theorem}[section]
\newtheorem{lemma}[theorem]{Lemma}
\newtheorem{proposition}[theorem]{Proposition}
\newcommand{\ZZ}{\mathbb Z}
\newcommand{\Om}{\Omega}
\newcommand{\cB}{\mathcal B}
\title{Multicolor Ramsey numbers of unbalanced double stars\footnote{Center for Discrete Mathematics, Fuzhou University,
Fuzhou 350108, P.~R.~China. Email: {\tt 1539166573@qq.com}. }}
\author{Lin Niu}
\date{}
\hypersetup{pdftitle={Multicolor Ramsey numbers of unbalanced double stars},
 pdfauthor={Lin Niu},
 pdfsubject={Exact multicolor Ramsey numbers of unbalanced double stars},
 pdfkeywords={Ramsey number, double star, edge-coloring, monochromatic component}}
\begin{document}
\maketitle
\begin{abstract}
For integers $n\ge m\ge1$, let $S(n,m)$ be the double star obtained
by joining the centers of $K_{1,n}$ and $K_{1,m}$.
Ruotolo and Song proved the upper bound
$r(S(n,m);k)\le kn+m+2$
under the condition
$(n+1)\lceil\frac{n+1}{k-1}\rceil
>m((k-1)n+m)$,
as well as the matching lower bound for odd $k$.
Zhao, Mao and Zhou proved the matching lower bound
for even $k$ and even $n$.
We establish the remaining lower bound for even $k\ge6$
and odd $n$ through a new edge-coloring construction.
Combining these results, we obtain
$r(S(n,m);k)=kn+m+2$
for all integers $k\ge5$ and $n\ge m\ge1$ satisfying
$(n+1)\lceil\frac{n+1}{k-1}\rceil
>m((k-1)n+m)$.
We also determine the four-color Ramsey number for
$m\ge2$ and odd $n\ge15m$, showing that
$r(S(n,m);4)=4n+m+1$.

\end{abstract}

\section{Introduction}\label{sec:intro}
All graphs are finite, simple, and undirected. Copies need not be induced.
For a graph $G$, write $V(G)$ and $E(G)$ for its vertex and edge sets,
and $N_G(v)$ and $d_G(v)$ for the neighborhood and degree of $v$.
Put $N_G[v]=N_G(v)\cup\{v\}$.
Let $G[X]$ denote the subgraph induced by $X\subseteq V(G)$.
A $k$-edge-coloring is a map from $E(G)$ to $[k]=\{1,\ldots,k\}$.
The $k$-color Ramsey number $r(H;k)$ is the least integer $N$ such that
every $k$-edge-coloring of $K_N$ contains a monochromatic copy of $H$.

For $n\ge m\ge1$, the double star $S(n,m)$ is formed by joining
the centers of disjoint stars with $n$ and $m$ leaves.
It has $n+m+2$ vertices and maximum degree $n+1$.
Grossman, Harary and Klawe~\cite{GHK} initiated the two-color problem
and determined exact values in several ranges.
Norin, Sun and Zhao~\cite{NSZ} disproved their general conjecture.
Further two-color bounds were obtained by Flores Dub\'o and Stein~\cite{FDS}.

Burr and Roberts~\cite{BR} determined the multicolor Ramsey numbers of stars.
Irving~\cite{Irving} studied multicolor Ramsey numbers of small graphs,
including $S(1,1)=P_4$.
Gy\'arf\'as and S\'ark\"ozy~\cite{GS} obtained bounds on the order of a largest
monochromatic double star in a multicolored complete graph.
S\'ark\"ozy~\cite{Sarkozy} later improved these bounds.
DeCamillis and Song~\cite{DCS} studied the corresponding bipartite Ramsey problem.
Bal, DeBiasio and Oren-Dahan~\cite{BDO} obtained bounds for balanced double stars.

For general double stars, Ruotolo and Song proved the following lower bounds.

\begin{theorem}[{Ruotolo and Song~\cite[Theorem~2.1]{RS}}]\label{thm:RS-lower}
Let $k\ge2$ and $n\ge m\ge1$ be integers.
For odd $k$,
\[
 r(S(n,m);k)\ge kn+m+2.
\]
For even $k$,
\[
 r(S(n,m);k)\ge\max\{kn+1,(k-1)n+2m+2\}.
\]
\end{theorem}

They also obtained an upper bound that matches the odd-color lower bound.

\begin{theorem}[{Ruotolo and Song~\cite[Theorem~2.7]{RS}}]\label{thm:RS-upper}
Let $k\ge2$ and $n\ge m\ge1$ be integers. If
\[
 (n+1)\left\lceil\frac{n+1}{k-1}\right\rceil
 >m\bigl((k-1)n+m\bigr),
\]
then $r(S(n,m);k)\le kn+m+2$.
\end{theorem}

Zhao, Mao and Zhou supplied the matching lower bound when both $k$ and $n$ are even.
Together with Theorem~\ref{thm:RS-upper}, this yields their exact formula.

\begin{theorem}[{Zhao, Mao and Zhou~\cite[Corollary~2.5 and Theorem~1.2]{ZMZ}}]\label{thm:ZMZ-even}
Let $k\ge4$ and $n\ge m\ge1$ be integers, with $k$ and $n$ even. Then
\[
 r(S(n,m);k)\ge kn+m+2.
\]
Equality holds whenever
$(n+1)\lceil(n+1)/(k-1)\rceil>m((k-1)n+m)$.
\end{theorem}

Under the upper-bound condition, Theorems~\ref{thm:RS-lower}--\ref{thm:ZMZ-even}
settle the cases with odd $k$ or even $n$ for $k\ge3$.
We supply the lower bound for even $k\ge6$ and odd $n\ge3$.
Our construction applies to every $1\le m\le n$ in this parity range.
Together with the preceding results, it gives the following theorem.

\begin{theorem}\label{thm:main}
Let $k\ge5$ and $n\ge m\ge1$ be integers. If
\begin{equation}\label{eq:condition}
 (n+1)\left\lceil\frac{n+1}{k-1}\right\rceil
 >m\bigl((k-1)n+m\bigr),
\end{equation}
then
\(
 r(S(n,m);k)=kn+m+2.
\)
\end{theorem}

Four colors behave differently. Our next theorem gives the value for odd $n$
when $n$ is sufficiently large relative to $m$.

\begin{theorem}\label{thm:four}
Let $m\ge2$ and $n\ge15m$ be integers, with $n$ odd. Then
\[
 r(S(n,m);4)=4n+m+1.
\]
\end{theorem}

Section~\ref{sec:many} proves Theorem~\ref{thm:main}.
Section~\ref{sec:four} proves Theorem~\ref{thm:four}.
For a coloring, let $G_c$ be the spanning graph of color $c$ and
put $d_c(v)=d_{G_c}(v)$.
A coloring is $S(n,m)$-free if it contains no monochromatic copy of $S(n,m)$.

\section{Five or more colors}\label{sec:many}
The only case not covered by the preceding lower bounds is when $k\ge6$ is even and $n$ is odd. We handle this case by first constructing a small coloring with carefully controlled monochromatic components, and then enlarging it to the required order. The enlargement uses the following multipartite coloring from~\cite[Lemma~2.4]{ZMZ}, with the vacuous case $s=0$ included.

\begin{lemma}[{\cite[proof of Lemma~2.4]{ZMZ}}]\label{lem:bulk}
Let $k\ge4$ and $s\ge0$ be even integers.
Let $U_c$, $c\in\Om$, be disjoint sets of size $s$, where $|\Om|=k$.
The complete $k$-partite graph on these sets has a $k$-edge-coloring
in which no edge of color $c$ meets $U_c$.
Every vertex in each other part has color-$c$ degree $s$.
\end{lemma}

The next step is to construct a core coloring on $3k+2$ vertices.
For each color $c$, the core contains a prescribed monochromatic
component $B_c$ of order five, while every vertex outside $B_c$
has color-$c$ degree at most three.
A suitable system of color permutations will then attach the sets
$U_c$ from Lemma~\ref{lem:bulk}, contributing the remaining $n-3$
neighbors in a controlled way.

\subsection{A coloring on \texorpdfstring{$3k+2$}{3k+2} vertices}
Fix an even $k\ge6$.
Let $\rho$ be an additional color label not in $\ZZ_{k-1}$, and set
$\Om=\ZZ_{k-1}\cup\{\rho\}$.
Write $h=2^{-1}$ in $\ZZ_{k-1}$.
Arithmetic on residue colors and subscripts is modulo $k-1$.
The inverse exists because $k-1$ is odd.

Let $H$ be the complete graph on $a_i,x_i,y_i$, $i\in\ZZ_{k-1}$,
and $b_1,\ldots,b_5$. Thus $|V(H)|=3k+2$.
Put $B_i=\{a_i,x_i,x_{i-1},y_i,y_{i-2}\}$ and
$B_\rho=\{b_1,\ldots,b_5\}$.
Also put $(\alpha_1,\ldots,\alpha_5)=(h,1,1,2,-h)$.
We color the edges of $H$ as follows.

For distinct $i,j$, the edges $a_i a_j$, $x_i x_j$, and $y_i y_j$
have colors $h(i+j)$, $h(i+j+1)$, and $h(i+j+2)$, respectively.
All edges inside $B_\rho$ have color $\rho$.

An edge $a_i x_j$ has color $h(i+j+1)$, with three exceptions.
The edge $a_i x_i$ has color $i$.
The edges $a_i x_{i+1}$ and $a_i x_{i-2}$ have color $\rho$.
For edges $a_i y_j$, the default color is $h(i+j+2)$.
It is replaced by $i$ at $j=i$, and by $\rho$ at $j=i+2$.

For edges $x_i y_j$, start with color $h(i+j+1)$.
At $j=i,i-2$, replace it by $i$; at $j=i+1,i-1$, replace it by $i+1$.
At $j=i-3$, replace it by $\rho$.
These five indices are distinct modulo $k-1$, since $k-1\ge5$.

Finally, the edge $a_i b_j$ has color $i+\alpha_j$.
In the order $b_1,\ldots,b_5$, the colors at $x_i$ are
$i+h,i+h,i+3h,i-h,i-1$.
The colors at $y_i$ are $i+1,i+1,i+1,i+h,i+3h$. Figure~\ref{fig:coloring} illustrates the color assignments
among the vertices $a_i$, $x_i$, and $y_i$, together with
the set $B_i$.
\begin{figure}[htbp]
\centering
\begin{tikzpicture}[
    x=1cm,y=0.84cm,
    every node/.style={font=\small},
    vertex/.style={circle,fill=black,draw=black,
                   inner sep=0pt,minimum size=4.4pt},
    edgelabel/.style={fill=white,inner sep=1pt,font=\scriptsize},
    paneltitle/.style={font=\small,anchor=base},
    rulezero/.style={red,line width=0.8pt},
    ruleone/.style={blue,line width=0.8pt},
    ruletwo/.style={black,line width=1.2pt}
]

\begin{scope}
    \coordinate (ai) at (-1.75,2.40);
    \coordinate (aj) at ( 1.75,2.40);
    \coordinate (xi) at (-1.75,1.55);
    \coordinate (xj) at ( 1.75,1.55);
    \coordinate (yi) at (-1.75,0.70);
    \coordinate (yj) at ( 1.75,0.70);

    \draw[rulezero] (ai)--(aj);
    \draw[ruleone] (xi)--(xj);
    \draw[ruleone] (ai)--(xj);
    \draw[ruleone] (xi)--(aj);
    \draw[ruleone] (xi)--(yj);
    \draw[ruleone] (yi)--(xj);
    \draw[ruletwo] (yi)--(yj);
    \draw[ruletwo] (ai)--(yj);
    \draw[ruletwo] (yi)--(aj);

    \node[vertex,label=left:$a_i$]  at (ai) {};
    \node[vertex,label=right:$a_j$] at (aj) {};
    \node[vertex,label=left:$x_i$]  at (xi) {};
    \node[vertex,label=right:$x_j$] at (xj) {};
    \node[vertex,label=left:$y_i$]  at (yi) {};
    \node[vertex,label=right:$y_j$] at (yj) {};

    \draw[rulezero] (-1.00, 0.18)--(-0.45, 0.18);
    \node[anchor=west,font=\scriptsize] at (-0.20, 0.18)
        {$h(i+j)$};
    \draw[ruleone] (-1.00,-0.18)--(-0.45,-0.18);
    \node[anchor=west,font=\scriptsize] at (-0.20,-0.18)
        {$h(i+j+1)$};
    \draw[ruletwo] (-1.00,-0.54)--(-0.45,-0.54);
    \node[anchor=west,font=\scriptsize] at (-0.20,-0.54)
        {$h(i+j+2)$};

    \node[paneltitle, yshift=-2mm] at (0,-1.04) {(a) Default color assignments};
\end{scope}
%
\begin{scope}[xshift=8cm]
    \coordinate (ba)  at ( 0,    1.10);
    \coordinate (bx0) at ( 0,    2.60);
    \coordinate (bxp) at (-1.80, 1.55);
    \coordinate (bxm) at ( 1.80, 1.55);
    \coordinate (by0) at (-1.10,-0.10);
    \coordinate (byp) at ( 1.10,-0.10);

    \draw (ba)--node[edgelabel,right,pos=.62] {$i$}    (bx0);
    \draw (ba)--node[edgelabel,above,pos=.62] {$\rho$} (bxp);
    \draw (ba)--node[edgelabel,above,pos=.62] {$\rho$} (bxm);
    \draw (ba)--node[edgelabel,left,pos=.65]  {$i$}    (by0);
    \draw (ba)--node[edgelabel,right,pos=.65] {$\rho$} (byp);

    \node[vertex,label=below:$a_i$]     at (ba)  {};
    \node[vertex,label=above:$x_i$]     at (bx0) {};
    \node[vertex,label=left:$x_{i+1}$]  at (bxp) {};
    \node[vertex,label=right:$x_{i-2}$] at (bxm) {};
    \node[vertex,label=below:$y_i$]     at (by0) {};
    \node[vertex,label=below:$y_{i+2}$] at (byp) {};

    \node[paneltitle, yshift=-2mm] at (0,-1.04) {(b) Exceptions at $a_i$};
\end{scope}
\end{tikzpicture}

\vspace{0.6em}

\begin{tikzpicture}[
    x=1cm,y=0.84cm,
    every node/.style={font=\small},
    vertex/.style={circle,fill=black,draw=black,
        inner sep=0pt,minimum size=4.4pt},
    edgelabel/.style={fill=white,inner sep=1pt,
        font=\scriptsize},
    paneltitle/.style={font=\small,anchor=base},
    rulezero/.style={red,line width=0.8pt},
    ruleone/.style={blue,line width=0.8pt},
    ruletwo/.style={black,line width=1.2pt}
]
\begin{scope}
    \coordinate (cx)  at ( 0,    1.10);
    \coordinate (cy0) at ( 0,    2.60);
    \coordinate (cym) at (-1.80, 1.55);
    \coordinate (cyp) at ( 1.80, 1.55);
    \coordinate (cy1) at (-1.10,-0.10);
    \coordinate (cy3) at ( 1.10,-0.10);

    \draw (cx)--node[edgelabel,right,pos=.62] {$i$}    (cy0);
    \draw (cx)--node[edgelabel,above,pos=.62] {$i$}    (cym);
    \draw (cx)--node[edgelabel,above,pos=.62] {$i+1$}  (cyp);
    \draw (cx)--node[edgelabel,left,pos=.65]  {$i+1$}  (cy1);
    \draw (cx)--node[edgelabel,right,pos=.65] {$\rho$} (cy3);

    \node[vertex,label=below:$x_i$]     at (cx)  {};
    \node[vertex,label=above:$y_i$]     at (cy0) {};
    \node[vertex,label=left:$y_{i-2}$]  at (cym) {};
    \node[vertex,label=right:$y_{i+1}$] at (cyp) {};
    \node[vertex,label=below:$y_{i-1}$] at (cy1) {};
    \node[vertex,label=below:$y_{i-3}$] at (cy3) {};

    \node[paneltitle, yshift=-2mm] at (0,-1.04) {(c) Exceptions at $x_i$};
\end{scope}

\begin{scope}[xshift=8cm]
    \coordinate (da)  at ( 0,    2.60);
    \coordinate (dx0) at (-1.70, 1.40);
    \coordinate (dxm) at ( 1.70, 1.40);
    \coordinate (dy0) at (-1.05,-0.10);
    \coordinate (dym) at ( 1.05,-0.10);

    \draw (da)--(dx0) (da)--(dxm) (da)--(dy0) (da)--(dym);
    \draw (dx0)--(dxm) (dy0)--(dym);
    \draw (dx0)--(dy0) (dx0)--(dym);
    \draw (dxm)--(dy0) (dxm)--(dym);

    \node[vertex,label=above:$a_i$]     at (da)  {};
    \node[vertex,label=left:$x_i$]      at (dx0) {};
    \node[vertex,label=right:$x_{i-1}$] at (dxm) {};
    \node[vertex,label=below:$y_i$]     at (dy0) {};
    \node[vertex,label=below:$y_{i-2}$] at (dym) {};

    \node[font=\small] at (0,1.02) {color $i$};
    \node[paneltitle, yshift=-2mm] at (0,-1.04) {(d) The set $B_i$};
\end{scope}
\end{tikzpicture}
\caption{Default color assignments in (a), exceptional assignments
in (b) and (c), and the monochromatic clique on $B_i$ in (d).}
\label{fig:coloring}
\end{figure}
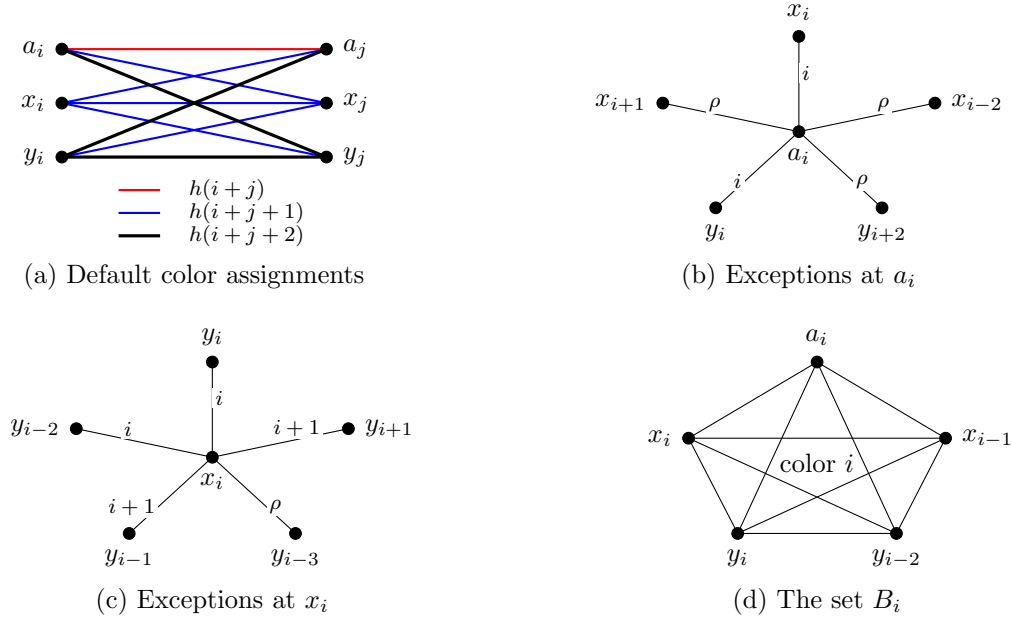

\begin{lemma}\label{lem:base}
For each $c\in\Om$, the set $B_c$ is a complete component of color $c$.
Every vertex outside $B_c$ has color-$c$ degree at most three.
More precisely, $d_{i+h}(x_i)=2$ and $d_\rho(y_i)=2$.
All other color degrees outside the specified components are three.
\end{lemma}

\begin{proof}
Put
\[
 A=\{a_i:i\in\ZZ_{k-1}\},\qquad
 X=\{x_i:i\in\ZZ_{k-1}\},\qquad
 Y=\{y_i:i\in\ZZ_{k-1}\}.
\]
Fix $i\in\ZZ_{k-1}$.
Since $h$ is invertible in $\ZZ_{k-1}$, each map
$j\mapsto h(i+j+s)$, $s\in\{0,1,2\}$, is a permutation
of $\ZZ_{k-1}$. For $c\in\ZZ_{k-1}\setminus\{i\}$, the coloring gives
$a_i$ exactly one color-$c$ neighbor in each of
\[
 A\setminus\{a_i\},\qquad
 X\cup\{b_1,b_2,b_5\},\qquad
 Y\cup\{b_3,b_4\}.
\]
Thus $d_c(a_i)=3$.
Counting the neighbors in these three sets also gives
\[
 d_i(a_i)=0+2+2=4,\qquad
 d_\rho(a_i)=0+2+1=3.
\]

For $c\in\ZZ_{k-1}\setminus\{i,i+1,i+h\}$, the vertex
$x_i$ has exactly one color-$c$ neighbor in each of
\[
 A\cup\{b_1,b_3\},\qquad
 X\setminus\{x_i\},\qquad
 Y\cup\{b_2,b_4,b_5\}.
\]
Hence $d_c(x_i)=3$.
For the other colors, the same three sets give
\[
\begin{aligned}
 d_i(x_i)=d_{i+1}(x_i)&=1+1+2=4,\\
 d_{i+h}(x_i)&=1+0+1=2,\\
 d_\rho(x_i)&=2+0+1=3.
\end{aligned}
\]

For $c\in\ZZ_{k-1}\setminus\{i,i+2\}$, the vertex
$y_i$ has exactly one color-$c$ neighbor in each of
\[
 A\cup\{b_1\},\qquad
 (Y\setminus\{y_i\})\cup\{b_2\},\qquad
 X\cup\{b_3,b_4,b_5\}.
\]
Therefore $d_c(y_i)=3$, while
\[
 d_i(y_i)=d_{i+2}(y_i)=1+1+2=4,\qquad
 d_\rho(y_i)=1+0+1=2.
\]

For each $j\in\{1,\ldots,5\}$ and $c\in\ZZ_{k-1}$,
the vertex $b_j$ has exactly one color-$c$ neighbor
in each of $A,X,Y$.
Its four edges to the other vertices of $B_\rho$
have color $\rho$.
Thus $d_c(b_j)=3$ and $d_\rho(b_j)=4$.

By the definition of the coloring,
$G_c[B_c]\cong K_5$ for every $c\in\Om$.
The degree counts above give $d_c(v)=4$ for every $v\in B_c$.
Hence no edge of color $c$ joins $B_c$ to its complement,
so $B_c$ is a complete component of color $c$.
\end{proof}

\subsection{Permutations of the colors}
We next specify the colors between $H$ and the new vertex classes.
For each $v\in V(H)$, we will use a permutation $\pi_v$ of the colors.
The condition $\pi_v(c)=c$ if and only if $v\in B_c$ will ensure that
$B_c$ joins only the new class $U_c$ in color $c$ and that no vertex
outside $B_c$ joins $U_c$ in that color. We also require that, for any
distinct colors $c,d$, exactly three vertices $v$ satisfy $\pi_v(c)=d$.
Thus each vertex of $U_c$ will have exactly three color-$d$ neighbors in $H$,
which, together with the $n-3$ supplied by Lemma~\ref{lem:bulk},
will give color-$d$ degree $n$.

Start with the permutation of $\Om$ that doubles elements of $\ZZ_{k-1}$ and fixes $\rho$.
Define $f_a$ by changing the images of $-h,\rho$ to $\rho,-1$.
Define $f_x$ by changing the images of $1,h,\rho$ to $1,\rho,2$.
Define $f_y$ by changing the images of $2,1,\rho$ to $2,\rho,4$.
In each list the images are given in the order of the inputs.
All other images remain unchanged.

Extend each translation of $\ZZ_{k-1}$ to $\Om$ by fixing $\rho$.
With this convention, for $T\in\{a,x,y\}$ put
\begin{equation}\label{eq:permutations}
 \pi_{T_i}(c)=i+f_T(c-i).
\end{equation}
For $b_j$, put $\pi_{b_j}(c)=c+\alpha_j$ for $c\in\ZZ_{k-1}$
and $\pi_{b_j}(\rho)=\rho$.

\begin{lemma}\label{lem:permutations}
Each $\pi_v$ is a permutation of $\Om$ that fixes $c$ exactly when $v\in B_c$.
For distinct $c,d\in\Om$, exactly three vertices $v$ satisfy $\pi_v(c)=d$.
\end{lemma}

\begin{proof}
The changes defining $f_a,f_x,f_y$ only permute the specified
images of the original permutation.
Hence all three maps are permutations of $\Om$.
It follows from their definitions that all $\pi_v$
are also permutations of $\Om$.

For each $i\in\ZZ_{k-1}$, the fixed-point sets of
$\pi_{a_i},\pi_{x_i},\pi_{y_i}$ are
$\{i\}$, $\{i,i+1\}$, and $\{i,i+2\}$, respectively.
Since every $\alpha_j$ is nonzero, $\pi_{b_j}$ fixes only $\rho$.
Thus $\pi_v(c)=c$ if and only if $v\in B_c$.

Now fix distinct $c,d\in\ZZ_{k-1}$.
For $T\in\{a,x,y\}$, we have
\[
 \pi_{T_i}(c)=i+2(c-i)=2c-i
\]
unless $c-i$ is one of the modified arguments of $f_T$.
At those arguments, $\pi_{T_i}(c)$ is either $c$ or $\rho$,
and therefore cannot equal $d$.
Consequently, any solution of $\pi_{T_i}(c)=d$ has
$i=2c-d$.
For this index, the definitions give
\[
\begin{aligned}
 \pi_{a_i}(c)=d
 &\quad\Longleftrightarrow\quad d-c\ne-h,\\
 \pi_{x_i}(c)=d
 &\quad\Longleftrightarrow\quad d-c\notin\{1,h\},\\
 \pi_{y_i}(c)=d
 &\quad\Longleftrightarrow\quad d-c\notin\{1,2\}.
\end{aligned}
\]

Since $\alpha_5=-h$, there is exactly one solution of
$\pi_v(c)=d$ with
$v\in\{a_j:j\in\ZZ_{k-1}\}\cup\{b_5\}$.
Similarly, $\alpha_1=h$ and $\alpha_2=1$ give exactly one
solution with
$v\in\{x_j:j\in\ZZ_{k-1}\}\cup\{b_1,b_2\}$,
while $\alpha_3=1$ and $\alpha_4=2$ give exactly one with
$v\in\{y_j:j\in\ZZ_{k-1}\}\cup\{b_3,b_4\}$.
These three sets partition $V(H)$.
Hence exactly three vertices satisfy $\pi_v(c)=d$.

Finally, if $c=\rho$ and $d\in\ZZ_{k-1}$, the definitions
give exactly the three solutions
\[
 v=a_{d+1},\qquad v=x_{d-2},\qquad v=y_{d-4}.
\]
If $c\in\ZZ_{k-1}$ and $d=\rho$, the three solutions are
\[
 v=a_{c+h},\qquad v=x_{c-h},\qquad v=y_{c-1}.
\]
\end{proof}

\subsection{The lower bound}
\begin{proposition}\label{prop:lower}
Let $k\ge6$ be even, let $n\ge3$ be odd, and let $1\le m\le n$.
There is an $S(n,m)$-free $k$-edge-coloring of $K_{kn+m+1}$.
\end{proposition}

\begin{proof}
Take $H$ above, disjoint sets $U_c$ of size $n-3$ for $c\in\Om$,
and a set $W$ of size $m-1$.
The total order is $(3k+2)+k(n-3)+(m-1)=kn+m+1$.
Retain the coloring on $H$.
Since $n-3$ is nonnegative and even, apply Lemma~\ref{lem:bulk} to color the edges
between distinct sets $U_c$. Color all edges inside $U_c$ with $c$.
For $v\in V(H)$, give every edge from $v$ to $U_c$ color $\pi_v(c)$.
For $w\in W$, give its edges to $a_i,x_i,y_i$ color $i$,
its edges to the $b_j$ color $\rho$, and its edges to $U_c$ color $c$.
Give all edges inside $W$ color $\rho$.

Fix $c$ and put $D_c=B_c\cup U_c\cup W$, of order $n+m+1$.
No color-$c$ edge of $H$ leaves $B_c$.
For $v\in B_c$, the permutation $\pi_v$ has output $c$ only at input $c$.
For $v\notin B_c$, we have $\pi_v(c)\ne c$.
Hence no edge from $H$ to a set $U_i$ leaves $D_c$ in color $c$.
Lemma~\ref{lem:bulk} excludes such edges between the sets $U_i$.
For $c\in\ZZ_{k-1}$, the color-$c$ neighbors of $W$ lie in
$\{a_c,x_c,y_c\}\cup U_c\subseteq D_c$.
For $c=\rho$, they lie in $B_\rho\cup U_\rho\cup W=D_\rho$.
Thus no color-$c$ edge joins $D_c$ to its complement.

For $v\in V(H)\setminus B_c$, Lemma~\ref{lem:base} gives at most three
color-$c$ neighbors in $H$.
Exactly one input of $\pi_v$ gives output $c$, so the sets $U_i$ add $n-3$ such neighbors.
No edge from $V(H)\setminus B_c$ to $W$ has color $c$. Thus $d_c(v)\le3+(n-3)=n$.
For $v\in U_i$ with $i\ne c$, Lemma~\ref{lem:bulk} gives $n-3$ such neighbors
in other parts. Lemma~\ref{lem:permutations} gives three in $H$.
There are none in $U_i\cup W$, so $d_c(v)=n$.

A color-$c$ copy of $S(n,m)$ would lie entirely in $D_c$ or in its complement,
since $S(n,m)$ is connected and no color-$c$ edge joins these sets.
The set $D_c$ has only $n+m+1$ vertices.
The color-$c$ graph on its complement has maximum degree at most $n$.
Neither can contain $S(n,m)$, which has $n+m+2$ vertices and maximum degree $n+1$.
This holds for every color, so the coloring is $S(n,m)$-free.
\end{proof}

\Needspace{6\baselineskip}
\begin{proof}[Proof of Theorem~\ref{thm:main}]
Theorem~\ref{thm:RS-upper} gives the upper bound.
The matching lower bound follows from
Theorem~\ref{thm:RS-lower} when $k$ is odd and from
Theorem~\ref{thm:ZMZ-even} when both $k$ and $n$ are even.

It remains to consider even $k\ge6$ and odd $n$.
If $n=1$, then $m=1$, and condition~\eqref{eq:condition}
would give $2>k$, a contradiction.
Hence $n\ge3$, and Proposition~\ref{prop:lower}
gives the matching lower bound.
\end{proof}

\section{The exceptional four-color case}\label{sec:four}
\subsection{The lower bound}
\begin{lemma}\label{lem:four-lower}
For odd $n\ge m\ge1$, there is an $S(n,m)$-free four-coloring of $K_{4n+m}$.
\end{lemma}

\begin{proof}
Partition the vertices into $A,B_1,B_2,B_3,B_4$,
where $|A|=m+1$, $|B_1|=|B_2|=|B_3|=n$, and $|B_4|=n-1$.
Split $B_4$ into $X,Y$ of size $(n-1)/2$.
Give edges inside $A$ color $4$.
For each $i$, give edges inside $B_i$ and between $A,B_i$ color $i$.
Read $i+1,i-1$ cyclically on $\{1,2,3\}$.
For $i\in\{1,2,3\}$, give edges between $X,B_i$ color $i+1$
and those between $Y,B_i$ color $i-1$.
For distinct $i,j\in\{1,2,3\}$, decompose the complete bipartite graph
between $B_i,B_j$ into $n$ perfect matchings, using label differences in $\ZZ_n$.
Color $(n-1)/2$ matchings with $4$ and the remaining $(n+1)/2$ with the color in
$\{1,2,3\}\setminus\{i,j\}$.

The set $A\cup B_4$ is a complete color-$4$ component of order $n+m$.
Each vertex outside it has $(n-1)/2$ color-$4$ neighbors in each of two other large parts.
Its color-$4$ degree is therefore $n-1$.
For $c\in\{1,2,3\}$, the set $A\cup B_c$ is connected in color $c$,
has order $n+m+1$, and has no color-$c$ edge to its complement.
Outside it, a vertex of $X\cup Y$ has $n$ color-$c$ neighbors in one large part.
A vertex in either remaining large part has $(n-1)/2$ such neighbors in $B_4$
and $(n+1)/2$ in the other large part, giving degree $n$.
In every color, the specified component is too small to contain $S(n,m)$,
and all other components have maximum degree at most $n$.
\end{proof}

\subsection{Degrees and component orders}

If an $S(n,m)$-free graph $H$ has an edge $uv$ with
$d_H(u)\ge n+1$ and $d_H(v)\ge m+1$, then
\begin{equation}\label{eq:neighborhood}
 |N_H(u)\cup N_H(v)|\le n+m+1.
\end{equation}
Indeed, put $U=N_H(u)\setminus\{v\}$ and $W=N_H(v)\setminus\{u\}$.
If the inequality fails, then $|U\cup W|\ge n+m$.
Choose $m$ vertices of $W$, taking those in $W\setminus U$ first.
If $|W\setminus U|\ge m$, no vertex of $U$ is used.
Otherwise, $|U|-(m-|W\setminus U|)=|U\cup W|-m\ge n$ vertices of $U$ remain.
We can therefore choose disjoint leaf sets of sizes $n,m$, giving an $S(n,m)$.

For the rest of this section, let $n,m$ satisfy Theorem~\ref{thm:four}.
Suppose $K_N$ has an $S(n,m)$-free four-coloring, where $N=4n+m+1$.
Put $V=V(K_N)$. We have $n>2m+2$.

\begin{lemma}\label{lem:degrees}
For every $v\in V$ and color $c$,
\begin{equation}\label{eq:degrees}
 n-2m\le d_c(v)\le n+m.
\end{equation}
\end{lemma}

\begin{proof}
We use a high- and low-degree count as in~\cite[Section~2]{BDO}.
For each color $c$, let $F_c=\{v\in V:d_c(v)\ge n+m+1\}$
and $L_c=\{v\in V:d_c(v)\le m\}$.
Put $F=\bigcup_{c=1}^4 F_c$.
For $u\in F_c$, every color-$c$ neighbor $v$ belongs to $L_c$.
Otherwise, \eqref{eq:neighborhood} applies to $uv$.
The union in that inequality contains $N_{G_c}[u]$, whose order is at least $n+m+2$,
a contradiction.

We also have $L_c\subseteq F$ for every $c$.
Indeed, a vertex $v\in L_c\setminus F$ would satisfy
$4n+m=\sum_{d=1}^4 d_d(v)\le m+3(n+m)<4n+m$, since $n>3m$.
Suppose $F\ne\varnothing$.
Every vertex of $F$ belongs to at least one $F_c$ and at most three sets $L_c$.
Count the color-$c$ edges incident with $F_c$ at their two ends.
All their other endpoints lie in $L_c$ and have color-$c$ degree at most $m$.
Summing over the four colors gives
\[
 (n+m+1)|F|
 \le \sum_{c=1}^4\sum_{v\in F_c}d_c(v)
 \le m\sum_{c=1}^4|L_c|
 \le 3m|F|.
\]
This contradicts $n+m+1>3m$.
Hence $F=\varnothing$, so $d_c(v)\le n+m$ for every $v$ and $c$.
Subtracting the other three upper bounds from $\sum_c d_c(v)=4n+m$
gives $d_c(v)\ge n-2m$.
\end{proof}

Call a monochromatic component \emph{high-degree} if it contains a vertex
of degree at least $n+1$ in its color.

\begin{lemma}\label{lem:orders}
Every high-degree component has $n+m+1$ vertices.
\end{lemma}

\begin{proof}
Take $u$ with $d_c(u)\ge n+1$ and put $D=N_{G_c}[u]$.
By~\eqref{eq:degrees}, $n+2\le|D|\le n+m+1$.
Every vertex has color-$c$ degree at least $n-2m\ge m+1$.
For $x\in N_{G_c}(u)$, we have $d_c(x)\ge n-2m\ge m+1$.
Applying~\eqref{eq:neighborhood} to the edge $ux$ gives
\[
|N_{G_c}(u)\cup N_{G_c}(x)|\le n+m+1.
\]
Since $ux\in E(G_c)$, we have $
 D=N_{G_c}(u)\cup\{u\}
 \subseteq N_{G_c}(u)\cup N_{G_c}(x)$.
Consequently,
\[
 |N_{G_c}(x)\setminus D|
 \le n+m+1-|D|.
\]
The same bound holds for $x=u$, whose color-$c$ neighbors all lie in $D$.
Thus each $x\in D$ has at least $N-|D|-(n+m+1-|D|)=3n$ neighbors
outside $D$ in the other three colors combined.

For $d\ne c$, let $Z_d$ consist of the vertices $x\in D$ with at least $n$
color-$d$ neighbors outside $D$ and with $d_d(x)\ge n+1$.
Each $y\notin D$ has at most $m+1$ color-$d$ neighbors in $Z_d$.
Suppose instead that $y$ has at least $m+2$ such neighbors, and choose one of them, $x$.
We have $d_d(x)\ge n+1$ and $d_d(y)\ge m+2$.
The sets $N_{G_d}(x)\setminus D$ and $N_{G_d}(y)\cap D$ are disjoint.
Thus~\eqref{eq:neighborhood} is contradicted by
\[
 |N_{G_d}(x)\cup N_{G_d}(y)|
 \ge |N_{G_d}(x)\setminus D|+|N_{G_d}(y)\cap D|
 \ge n+m+2.
\]

Counting color-$d$ edges between $Z_d$ and $V\setminus D$ gives
$n|Z_d|\le(m+1)(N-|D|)$.
Since $|D|\ge n+2$, it follows that
\begin{equation}\label{eq:exceptional}
 \left|\bigcup_{d\ne c}Z_d\right|
 \le\frac{3(m+1)(3n+m-1)}{n}<15m\le n.
\end{equation}
Indeed, $3(3n+m-1)<10n$ and $10(m+1)\le15m$.
Choose $v\in D\setminus\bigcup_{d\ne c}Z_d$, since $|D|\ge n+2$.
For each $d\ne c$, the vertex $v$ has at most $n$ color-$d$ neighbors
outside $D$: otherwise its total color-$d$ degree would also be at least
$n+1$, placing it in $Z_d$.
These three outside degrees sum to at least $3n$, so each equals $n$.
Since the definition of $Z_d$ uses an outside degree of at least $n$,
the fact that $v\notin Z_d$ now implies $d_d(v)\le n$.
Consequently $d_d(v)=n$ for every $d\ne c$, and
$d_c(v)=4n+m-3n=n+m$.

Put $D'=N_{G_c}[v]$, so $|D'|=n+m+1$.
For every color-$c$ neighbor $w$ of $v$, we have $d_c(w)\ge m+1$ and
$D'\subseteq N_{G_c}(v)\cup N_{G_c}(w)$.
Equation~\eqref{eq:neighborhood} bounds this union by $n+m+1$,
so it equals $D'$. In particular, $N_{G_c}(w)\subseteq D'$.
The color-$c$ neighbors of $v$ also lie in $D'$ by definition.
Thus no color-$c$ edge leaves $D'$, and $G_c[D']$ is connected because
$v$ is adjacent in color $c$ to every other vertex of $D'$.
Therefore $D'$ is the component of $v$.
Since $v\in N_{G_c}[u]$, it is also the component of $u$.
\end{proof}

\subsection{Component intersections}
Let $\cB$ be the family of all high-degree monochromatic components,
each identified with its vertex set and color. Put $r=|\cB|$.
For $v\in V$, let $\sigma(v)$ be the set of colors of the members of $\cB$ containing $v$.
We call $\sigma(v)$ the \emph{support} of $v$.
Its size counts those members, since components of the same color are disjoint.
A vertex is \emph{private} if it belongs to exactly one member.
Vertices of support size two, three, or four are called \emph{double},
\emph{triple}, or \emph{quadruple}, respectively.

If $c\notin\sigma(v)$, then $d_c(v)\le n$; otherwise the color-$c$
component containing $v$ would be high-degree, giving $c\in\sigma(v)$.
Conversely, membership in a high-degree component does not require
$v$ itself to have degree at least $n+1$ in that color.
Every support is nonempty, since the color degrees sum to $4n+m>4n$.
Thus the members of $\cB$ cover $V$.
Three components cannot cover $V$, as $3(n+m+1)<N$.
No color has three or more high-degree components.
Indeed, removing any three such components would leave
$N-3(n+m+1)=n-2m-2>0$ vertices.
None of their edges in that color could enter the removed components,
so their degrees in that color would be at most $n-2m-3$,
contrary to~\eqref{eq:degrees}.
Thus every color has at most two high-degree components, and $4\le r\le8$.

If $u$ is private to a component $B$ of color $d$, then
$d_d(u)\le n+m$ by~\eqref{eq:degrees}, and $d_c(u)\le n$ for $c\ne d$
by the preceding observation.
The sum of these four upper bounds is $4n+m$, the total degree of $u$,
so all four are attained: $d_d(u)=n+m$ and $d_c(u)=n$ for $c\ne d$.
All color-$d$ neighbors of $u$ lie in $B$, and $|B\setminus\{u\}|=n+m$.
Hence every edge from $u$ to $B\setminus\{u\}$ has color $d$.

No two vertices have complementary supports.
Indeed, in each color one endpoint would lie in a high-degree component
and the other outside all such components.
Their edge could receive no color.

\begin{lemma}\label{lem:large-count}
There are exactly four high-degree components.
\end{lemma}

\begin{proof}
Let $B,C\in\cB$ be distinct components with nonempty intersection,
and take $v\in B\cap C$. Their colors, say $a,b$, are distinct.
Put $s=|\sigma(v)|$ and $I=|B\cap C|$.
The color-$a$ and color-$b$ neighbor sets of $v$ lie in $B$ and $C$,
respectively, and are disjoint. Therefore
\[
 d_a(v)+d_b(v)
 \le |(B\cup C)\setminus\{v\}|=2(n+m)+1-I.
\]
Among the remaining colors, the $s-2$ colors in $\sigma(v)$ have degree
at most $n+m$, while the $4-s$ colors outside $\sigma(v)$ have degree
at most $n$. Consequently,
\[
\begin{aligned}
 4n+m=\sum_{d=1}^4d_d(v)
 &\le 2(n+m)+1-I+(s-2)(n+m)+(4-s)n\\
 &=4n+sm-(I-1).
\end{aligned}
\]
It follows that
\begin{equation}\label{eq:intersection-bound}
 |B\cap C|\le (|\sigma(v)|-1)m+1
 \qquad\text{for every }v\in B\cap C.
\end{equation}
Since $|\sigma(v)|\le4$, we obtain $|B\cap C|\le3m+1$.
This bound also holds when $B$ and $C$ are disjoint.

Suppose $r\ge5$.
A quadruple vertex would belong to four components covering $V$:
the color of its edge to any other vertex places that vertex in one of them.
An additional component is disjoint from the selected component of its own color.
It meets each of the other three in at most $3m+1$ vertices.
Its order would then satisfy $n+m+1\le3(3m+1)$, contrary to $n\ge15m$.
Thus no quadruple vertex exists, so $|\sigma(v)|\le3$ for every $v$.
Equation~\eqref{eq:intersection-bound} now gives $|B\cap C|\le2m+1$
for all distinct $B,C\in\cB$.

Each $B\in\cB$ meets at most six other components, since each of the other three colors
has at most two high-degree components.
It therefore has at least $n+m+1-6(2m+1)=n-11m-5$ private vertices.
This is positive because $n-11m-5\ge4m-5>0$.

Choose any five high-degree components.
Their union has order at least the sum of their orders minus the sum of their pairwise intersections.
Indeed, a vertex in $j\ge1$ of the five sets contributes $j-\binom j2\le1$ to this difference.

Suppose first that all four colors have high-degree components.
A triple vertex would have support complementary to a private vertex in the missing color.
Hence there are no triple vertices either, so every support has size at most two.
Equation~\eqref{eq:intersection-bound} gives $|B\cap C|\le m+1$.
Among the five chosen components, at least two have the same color and are disjoint.
At most nine pairs therefore intersect, so their union has order at least
\[
 5(n+m+1)-9(m+1)=5n-4m-4>N,
\]
since $n-5m-5\ge10m-5>0$.
This is a contradiction.

Otherwise, high-degree components occur in at most three colors.
Among the five chosen components, at least two pairs have the same color.
Thus at most eight pairs intersect, and their union has order at least
\[
 5(n+m+1)-8(2m+1)=5n-11m-3>N,
\]
since $n-12m-4\ge3m-4>0$.
This is also a contradiction, so $r\le4$.
Together with $r\ge4$, this proves the lemma.
\end{proof}

\subsection{Completion of the upper bound}

\begin{lemma}\label{lem:components}
There is exactly one high-degree component in each color.
Each has a private vertex.
\end{lemma}

\begin{proof}
Lemma~\ref{lem:large-count} gives $r=4$.
Counting incidences yields
\begin{equation}\label{eq:excess}
 \sum_{v\in V}(|\sigma(v)|-1)=4(n+m+1)-N=3(m+1).
\end{equation}
Every nonprivate vertex contributes at least one.
Each component therefore has at least $n+m+1-3(m+1)=n-2m-2>0$ private vertices.
Suppose color $1$ has two components $A_1,A_2$.
The other two components, $B,C$, have colors different from $1$.
Put $R=V\setminus(A_1\cup A_2)$, so $|R|=2n-m-1$ and $R\subseteq B\cup C$.
A private vertex $u$ of $B$ has $n$ color-$1$ neighbors.
None lies in $A_1\cup A_2$, since those are color-$1$ components.
None lies in $B$, since every edge from $u$ to the rest of $B$ has the color of $B$.
Thus $n\le|R\setminus B|$, giving $|R\cap B|\le n-m-1$.
The same argument gives $|R\cap C|\le n-m-1$.
It follows that
\[
 2n-m-1=|R|\le|R\cap B|+|R\cap C|\le2n-2m-2,
\]
a contradiction. The four components have distinct colors.
\end{proof}

\begin{proof}[Proof of Theorem~\ref{thm:four}]
Lemma~\ref{lem:four-lower} gives the lower bound.
For the upper bound, use the coloring fixed above.
Its high-degree components are $B_1,B_2,B_3,B_4$, with $B_c$ of color $c$.
Each has a private vertex by Lemma~\ref{lem:components}.
There is no triple vertex, since its support would be complementary
to that of a private vertex in the missing color.

Let $t$ and $D_2$ count quadruple and double vertices, respectively.
Since there are no triple vertices, equation~\eqref{eq:excess} gives
$D_2+3t=3(m+1)$, and hence $D_2=3(m+1-t)$.
The six two-element subsets of $\{1,2,3,4\}$ form three complementary pairs.
At most one support from each pair can occur, so there are at most three
distinct double-vertex supports.
If $v$ is a double vertex with support $\{a,b\}$, its color-$a$ and color-$b$
neighbor sets are disjoint subsets of
$(B_a\cup B_b)\setminus\{v\}$. Hence
\[
 d_a(v)+d_b(v)\le2(n+m)+1-|B_a\cap B_b|.
\]
The other two color degrees are at most $n$, so
$4n+m\le4n+2m+1-|B_a\cap B_b|$.
Thus $|B_a\cap B_b|\le m+1$.
This intersection consists of the double vertices with support $\{a,b\}$
and the $t$ quadruple vertices, since no triple vertex exists.
There are therefore at most $m+1-t$ double vertices with this support.
Suppose first that $D_2>0$, so $m+1-t>0$, and let $p\le3$ be the
number of double-vertex supports that occur. Summing the preceding bound
over these supports gives
\[
 3(m+1-t)=D_2\le p(m+1-t)\le3(m+1-t).
\]
Thus $p=3$, and each of the three supports occurs exactly $m+1-t$ times.
For every occurring support $\{a,b\}$, we consequently have
$|B_a\cap B_b|=(m+1-t)+t=m+1$.
The preceding neighbor bound gives $d_a(v)+d_b(v)\le2n+m$.
Its other two degrees sum to at least $2n$ and are each at most $n$,
so both equal $n$.
Private vertices also have degree $n$ in each color outside their support.
If $D_2=0$, every nonquadruple vertex is private, so the same conclusion holds.
In either case, every vertex has degree $n$ in each color outside its support.

Fix a color $c$.
Every vertex outside $B_c$ is private or double and has color-$c$ degree $n$.
No edge of color $c$ leaves $B_c$.
Thus $G_c[V\setminus B_c]$ is $n$-regular on $N-(n+m+1)=3n$ vertices.
Since $n$ is odd, the degree sum $3n^2$ is odd,
contradicting the handshaking lemma.
\end{proof}

\section*{Declaration on the use of AI}
During the development of the four-color upper bound for $S(n,m)$, we used ChatGPT (OpenAI) to discuss possible approaches to a structural question concerning the orders of monochromatic components containing vertices of high color degree. These discussions contributed to the investigation of the component-order property established in Lemma 3.3 and to the development of its proof.

ChatGPT was also used to improve the language and presentation of the manuscript. The author independently reviewed and verified all mathematical arguments and takes full responsibility for the results and conclusions.

\section*{Conflict of interest}

The author declares no conflict of interest.

\end{document}